\documentclass[12pt, reqno]{amsart}
\usepackage{amsmath,amsthm,amscd,amsfonts,amssymb,graphicx,color}
\usepackage[bookmarksnumbered,colorlinks,plainpages]{hyperref}

\hypersetup{colorlinks=true,linkcolor=black, anchorcolor=black,
	citecolor=black, urlcolor=black, filecolor=magenta, pdftoolbar=true}
\newtheorem{theorem}{Theorem}

\theoremstyle{definition}
\newtheorem{definition}{Definition}
\newtheorem{example}{Example}

\theoremstyle{remark}
\newtheorem{remark}{Remark}
\numberwithin{equation}{section}

\begin{document}
\setcounter{page}{1}


\begin{center}
{\textbf{\Large Strongly geodesic preinvexity and Strongly Invariant $\eta$-Monotonicity on Riemannian Manifolds and its Application}}\\
\bigskip
{\textbf {Akhlad Iqbal$^1$, Askar Hussain$^2$ and Hilal Ahmad Bhat$^3$ }}\\

Department of Mathematics,\\ Aligarh Muslim University, Aligarh-202002, India\\[0pt]
$^1$ akhlad6star@gmail.com,$^2$askarhussain59@gmail.com and $^3$bhathilal01@gmail.com
\end{center}
 \bigskip
 
\textbf{Abstract}:
 In this paper, we present strongly geodesic preinvexity on Riemannian manifolds (RM) and  strongly $\eta$-invexity of order $m$ on RM. Furthermore, we define strongly invariant $\eta$-monotonicity of order $m$ on RM. Under {\bf{Condition C}}, an important characterization of these functions are studied. We construct several non-trivial examples in support of these definitions.  Afterwords, an important and significant characterization of a strict $\eta$-minimizers ($\eta$-minimizers) of order $m$ for $MOP$ and a solution to the variational like-inequality problem
 $(VVLIP)$ has been derived. 

\section{\bf{Introduction}}
The variational inequality problem and monotonicity play a vital role in the existence of the solution to the practical problems in many fields of mathematical science and physics such as optimization theory, science, engineering, etc. The generalized monotonicity is effective equipment for the existence and analysis of a solution to the variational inclusion and complementarity problems. The convexity is closely related to monotonicity. It is observed that the monotonicity of the corresponding gradient function is equivalent to the convexity of the real-valued function. Recently, the extension of convexity and monotonicity have been developed by authors, see \cite{Pini2,Correa,Fan}.  Generally, a manifold is a different space from a linear space. Rapcsak \cite{Rapcsak} and Udriste \cite{Udriste} extended convexity with techniques from linear space to RM, which is called geodesic convexity. The different kinds of invariant monotone vector fields on RM presented by Barani \cite{Poury2}. It has been shown that many the results and properties established from the Euclidean space to Riemannian  manifolds preserves for invariant monotone vector fields. Also, optimization has been more developed on RM, see \cite{Yang, Garzon}. The concept of invex function on RM defined by Pini \cite{Pini1} and several properties have been discussed. The generalized invexity has developed the convex analysis, and the generalized invexity is closely related to the generalized invariant monotonicity, which has been studied in \cite{Lang}. Yang et al. \cite{Yang} presented invariant monotonicity, which is the generalization of monotonicity. The existence of a solution to $VLIP$ was derived under generalized monotonicity \cite{Lang}. Nemeth \cite{Nemeth}, presented a monotone vector field on RM, which is an important generalization monotone operator. Noor  \cite{Noor} discussed the notion of $\alpha$-invexity  and  $\alpha$-monotonicity.  Iqbal et al. \cite{Akhlad3} extended it on RM, which is called strong $\alpha$-invexity and invariant $\alpha$-monotonicity. 

Motivated by research works, see in \cite{Poury2,Hussain1,Hussain2,Noor,hanson,Akhlad3,Azagra}, we introduce the strongly geodesic preinvexity of order $m$, strongly $\eta$-invexity of order $m$,  strongly quasi $\eta$-invexity of order $m$ and strongly pseudo $\eta$-invexity of order $m$ on RM, which are generalization of strongly geodesic convex function of order $m$ defined by \cite{Akhlad1}, strongly $\alpha$-invexity and invariant $\alpha$-monotonicity defined by  Iqbal et al. \cite{Akhlad3}. The article is as follows: Section \ref{2a} contains few definitions and facts are undersigned, mainly our concentrate is to know the basic concept about Riemannian geometry which play the main role in this article. Nontrivial suitable example are constructed in support of these definitions and several interesting properties and results are proved. An important characterization of strongly geodesic preinvexity of order $m$ and strongly $\eta$-invexity of order $m$ has been introduced in Section \ref{3a}.

 The generalized invariant $\eta$-monotonicity of order $m$ such as strongly invariant $\eta$-monotonicity and strongly invariant pseudo $\eta$-monotonicity are defined on the RM. A relationship between the strongly $\eta$-invexity (strongly pseudo $\eta$-invexity) and strongly invariant $\eta$-monotonicity (strongly invariant pseudo $\eta$-monotonicity) has been established, which show the strongly $\eta$-invexity (strongly pseudo $\eta$-invexity) is closely related to strongly invariant $\eta$-monotonicity (strongly invariant pseudo $\eta$-monotonicity) in Section \ref{4a}. In Section \ref{5a}, we introduce $\eta$-minimizers of order $m$ on the RM. The $MOP$ for strongly $\eta$-invex function of order $m$ has been presented an application. A relationship between the strict $\eta$-minimizers of order $m$ for $MOP$ and a solution to $VVLIP$ has been introduced.
\section{\bf{Preliminaries}}\label{2a} 
In this section, we recall few definitions and basic results regarding RM, which will be used everywhere in this article. For the  standard material on differential geometry, consult, \cite{Lang}. Here, $M$ is considered as $C^{\infty}$ smooth manifold modelled on a Hilbert space $H$, either finite dimensional or  infinite dimensional, endowed with   Riemannian metric $\langle\cdot,\cdot\rangle _{p}$  on $T_{p}M$ at point $p\in M$, $T_{p}M\cong M$. Thus, we get 
a smooth assignment of $\langle\cdot,\cdot\rangle _{p}$ to every $T_{p}M$. Usually, we can write 
\begin{eqnarray*}
	g_{p}\left(X_{1},X_{2}\right)=\left\langle X_{1},X_{2}\right\rangle _{p}~\forall X_{1},X_{2}\in T_{p}M.
\end{eqnarray*}
Therefore, $M$ is a RM. The length of tangent vector corresponding norm of inner product $\langle.,.\rangle _{p}$ is denoted by $\|\cdot\|_{p}$. The length of piece wise $C^{1}$ curve $r:[a,b]\rightarrow M$ is defined by 
\begin{eqnarray*}
	L(r)=\int\limits_{a}^{b}\|r{'}(s)\|_{r(s)}ds.
\end{eqnarray*}
For $x_{1},x_{2}\in M,~~~ x_{1}\neq x_{2}$, we define the distance between $x_{1}$ and $x_{2}$.
\begin{eqnarray*}
	 d(x_{1},x_{2})=\inf\{L(r):r~ \mbox{is}~ \mbox{piecewise}~ C^{1}~ \mbox{curve} ~\mbox{joining}~ x_{1}~ \mbox{to}~ x_{2}\}.
\end{eqnarray*}
Then, the original topology on $M$ is induced by the distance $d$. The set of all vector fields over smooth manifold   $M$ is denoted by $\chi(M)$. The metric induces a map $h\rightarrow grad~h\in\chi(M)$ which associates with every $h$ its gradient. 
i.e.,
\begin{eqnarray*}
	\left\langle dh,X\right\rangle _{p}= dh(X),~\forall X\in\chi(M).
\end{eqnarray*}

It is known that on every RM, $\exists$ exactly one covariant derivation called Levi-Civita connection denoted by $\nabla_{X}Y$ for any $X,Y\in \chi(M)$. Again recall that a geodesic is a $C^{\infty}$ smooth path $r$ joining $x_{1},x_{2}\in M$ s.t. $L(r)=d(x_{1},x_{2}),~x_{1}\neq x_{2}$, is called a length of curve $C$,  it is called minimal geodesic which satisfies the equation $\nabla_{\frac{dr(t)}{dt}}\frac{dr(t)}{dt}=0$.  The existence theorem for ordinary differential equations states that $\forall v\in TM$, $\exists$ an open interval $J(v)$ containing $0$ and exactly one geodesic $r_{v}:J(v)\rightarrow M$ with $\frac{dr(0)}{dt}=v$. This implies that $\exists$ an open neighborhood $\widetilde{T}M$ of the manifold $M$ s.t.  $\forall~~v\in\widetilde{T}M$, the geodesic $r_{v}(t)$ is defined for $|t|<2$.
By using parallel translation operator of vectors along smooth curve, for a smooth curve $r:I\rightarrow M$, a vector $v_{0}\in T_{r(t_{0})}M,~~\forall~~ t_{0}\in I$, $\exists$ exactly one parallel vector field $V(t)$ along $r(t)$ s.t. $V(t_{0})=v_{0}$. The Hilbert space $(T_{p}M,\|\cdot\|_{p})$ is a linear isometric  identification between its dual space $(T_{p}^{*}M,\|\cdot\|_{p})$. The linear isometric map between  the tangent spaces $T_{r(t_{0})}M $ and $T_{r(t)}M,~~\forall~t\in I$ is defined by $v_{0}\rightarrow V(t)$ and denoted by $P_{t_{0},r}^{t}$ which call the parallel translation from  $T_{r(t_{0})}M$ to $T_{r(t)}M$ along   $r(t)$. There exists a linear map $dh_{p}:T_{p}M \rightarrow T_{h(p)}N$ if  $h$ is differentiable map from the manifold $M$ into the manifold $N$, where $dh_{p}$ represents the differential of $h$ at $p$.
A RM of finite-dimensional is called complete if its geodesic is defined for all values of $t$. Hopf-Rinow's theorem assures that the RM $M$ is complete if all pairs of points in $M$ can be joined by a (not necessarily unique) minimal geodesic segment.
$M$ considered RM, $\eta:M\times M\rightarrow TM$ be   map s.t.   $\forall~~u,v\in M,~\eta(u,v)\in TM$ and some basic definitions are   as follows: \\
 
 Pini \cite{Pini1} defined the definition is given as follows:
 \begin{definition}\label{7}\cite{Pini1}
 	Let $r_{u,v}:[0,1]\rightarrow M$ be a curve on RM $M$ s.t. $r_{u,v}(0)=v$ and $r_{u,v}(1)=u$. Then, curve  $r_{u,v}$ is called possess the property $(P)$ w.r.t. $v,u\in M$, if
 \begin{eqnarray*}
  r_{u,v}{'}(t)(s-t)=\eta(r_{u,v}(s),r_{u,v}(t)),~~~\forall~~~s,t\in[0,1].
 \end{eqnarray*}
 \end{definition}
 \begin{definition}\label{7}\cite{Pini1}
 	For RM $M$, a map $\eta:M\times M\rightarrow TM$ is called integrable if  $\forall~~~u,v\in M,~~\exists$ at least one curve $r_{u,v}$ possess the property $(P)$ w.r.t. $v,u\in M$.
 \end{definition}
 \begin{remark}\cite{Poury1}
 	Let $M$ be a RM. A map $\eta:M\times M\rightarrow TM$ is called integrable. If we get
 \begin{eqnarray*}
 	r_{u,v}{'}(0)=\eta(r_{u,v}(1),r_{u,v}(0))=\eta(u,v).
 \end{eqnarray*}
\end{remark} 
 	In the case, if $r_{u,v}(t)$ is a geodesic, then satisfy {\bf{Condition C}}
 	\begin{eqnarray*}
 		{\bf C_{1}} ~~\eta(v,r_{u,v}(s))=-sr_{u,v}(0)=  -sP_{0,r_{u,v}}^{s}\left[r_{u,v}{'}(0)\right] =-sP_{0,r_{u,v}}^{s}\left[\eta(u,v)\right] 
 	\end{eqnarray*} 
 \begin{eqnarray*}
 	 {\bf C_{2}} ~~\eta(r_{u,v}(1),r_{u,v}(s))=(1-s)r_{u,v}(0)&=& (1-s)P_{0,r_{u,v}}^{s}\left[r_{u,v}{'}(0)\right]\\&=&(1-s)P_{0,r_{u,v}}^{s}\left[\eta(u,v)\right].\hspace*{1.4cm}
 \end{eqnarray*}
Together ${\bf C_{1}}$ and ${\bf C_{2}}$ is called {\bf{Condition C}} defined by Barani et al. \cite{Poury1}.\\

The strongly geodesic convex of order $m$ defined by Akhlad et al. \cite{Akhlad1}.
\begin{definition}\cite{Akhlad1}
	A function $h:S\rightarrow R$ is called strongly geodesic convex of order $m$ on geodesic convex set $S$, if $\exists~~~\delta>0$ s.t. 
\begin{eqnarray*}
	 h(r_{u,v}(s))\leq s h(u)+(1-s)h(v)-\delta s (1-s)\|r{'}_{u,v}(s)\|^{m},
\end{eqnarray*} for every $u,v\in S$, $s\in[0,1]$.
 \end{definition}
The concept of geodesic invex sets was defined by Barani et al.\cite{Poury1}, which is given as follow:
\begin{definition}\cite{Poury1} \label{D6} 
   Let $M$ be a RM and $\eta:M\times M\rightarrow TM$ be map such that for each $u,v \in M, \eta(u,v)\in T_{v}M$. A  non empty set $S\subseteq M$ is called geodesic invex with respect to (w.r.t.) $\eta$ if $\exists$ exactly one geodesic $r_{u,v}:[0,1]\rightarrow M$ s.t.
\begin{eqnarray*}
	r_{u,v}(0)=v,~~~r_{u,v}{'}(0)=\eta(u,v),~~~r_{u,v}(s)\in S,~\forall~s\in[0,1],~u,v\in S.
\end{eqnarray*}
\end{definition}
\begin{definition} \cite{Poury1} 
	A function $h:S\rightarrow R$ is called geodesic $\eta$-preinvex on  geodesic invex set $S$, if $\forall~ u,v\in S, s\in[0,1]$, we have
\begin{eqnarray*}
	 h(r_{u,v}(s))\leq s h(u)+(1-s)h(v).
\end{eqnarray*}
\end{definition}
\begin{definition}\cite{Poury2}\label{8D}
	Let $M$ be a RM. A  differentiable function $h:M\rightarrow R$ is called strongly $\eta$-invex of order $2$ w.r.t. $\eta$, if $\exists~~\delta>0$ s.t.
	\begin{eqnarray*}
		h(u)\geq h(v)+dh_{v}(\eta(u,v))+\delta\|\eta(u,v)\|_{v}^{2},~~ \forall~ u,v\in M.
	\end{eqnarray*}
\end{definition}
\section{\bf{Strongly geodesic preinvexity on Riemannian manifolds}}\label{3a}
  In this section, we introduce strongly geodesic preinvexity on RM $M$ and strongly $\eta$-invexity of functions of order $m$ on RM $M$.
      \begin{definition}\label{4}
     Let $m $ be a positive integer. Let $S\subseteq M$ be a geodesic invex subset of RM $M$. A function  $h:S\rightarrow R$ is called strongly preinvex of order $m$ w.r.t. $\eta$, if $\exists~\delta>0$, s.t.  $\forall~u,v\in S$,  $s \in[0,1]$, we have
      \begin{eqnarray*}
      	h(r_{u,v}(s))\leq  s h(u)+(1-s)h(v)-\delta s (1-s)\left\|\eta(u,v)\right\|_{v}^{m},
      \end{eqnarray*}
      	 where $r_{u,v}(s)$ is the unique geodesic defined in Definition \ref{D6}.   
      \end{definition} 
      \begin{remark}
      	 If $\eta(u,v)=u-v,~\forall u,v\in S$, then it reduces to strongly geodesic convexity of order $m$, see \cite{Akhlad1}. 
      \end{remark}
      \begin{remark}
      	For $\delta=0$, Definition \ref{4} reduces to geodesic preinvex defined in \cite{Poury1}.
      \end{remark}
      \begin{theorem}
        	    If $h_{1},h_{2},...,h_{k}$, are strongly geodesic preinvex functions of order $m$ on geodesic invex set $S$, then $h=\sum\limits_{ j=1}^{k}a_{j}h_{j}$ and $H=\max\limits_{1\leq j\leq k}h_{j} $ are  also strongly geodesic preinvex functions of order $m$, where $a_{j}>0,~1\leq j\leq k$.
        \end{theorem}
        \begin{proof}
        Since $h_{j}, 1\leq j\leq k$, are strongly geodesic preinvex functions of order $m$, for every $u,v\in S,~s\in[0,1]$, we have
       \begin{eqnarray*}
       	h_{j}(r_{u,v}(s))\leq  s h_{j}(u)+(1-s)h_{j}(v)-\delta s (1-s)\|\eta(u,v)\|_{v}^{m},~1\leq j\leq k.
       \end{eqnarray*}
        Now taking $\sum\limits_{ j=1}^{k}a_{j}$ on both sides, we get
      \begin{eqnarray*}
      	\sum\limits_{ j=1}^{k}a_{j}h_{j}(r_{u,v}(s))\leq  s \sum\limits_{ j=1}^{k}a_{j} h_{j}(u)+(1-s) \sum\limits_{ j=1}^{k}a_{j}h_{j}(v)-\sum\limits_{ j=1}^{k}a_{j}\delta s (1-s)\left\|\eta(u,v)\right\|_{v}^{m},
      \end{eqnarray*}
        \begin{eqnarray*}
        	h(r_{u,v}(s))\leq  s h(u)+(1-s)h(v)-\delta^{'} s (1-s)\left\|\eta(u,v)\right\|_{v}^{m},
        \end{eqnarray*}
         choose $\delta^{'}=\sum\limits_{ j=1}^{k}a_{j}\delta>0.$
           Hence, the function $h$ is strongly geodesic preinvex   of order $m$.\\
          For other part, since $h_{j}, 1\leq j\leq k$, are strongly geodesic preinvex functions of order $m$, we have\begin{eqnarray*}
          	h_{j}\left(r_{u,v}(s)\right)\leq  s h_{j}(u)+(1-s)h_{j}(v)-\delta s (1-s)\left\|\eta(u,v)\right\|_{v}^{m},~\forall~1\leq j\leq k.
          \end{eqnarray*}
          Now taking $\max\limits_{1\leq j\leq k}h_{j}$, we get
          \begin{eqnarray*}
          	\max_{1\leq j\leq k}h_{j}\left(r_{u,v}(s)\right)\leq  s \max_{1\leq j\leq k}h_{j}(u)+(1-s)\max_{1\leq j\leq k}h_{j}(v)-\delta s (1-s)\left\|\eta(u,v)\right\|_{v}^{m}.
          \end{eqnarray*}
         \begin{eqnarray*}
         	H\left(r_{u,v}(s)\right)\leq  s H(u)+(1-s)H(v)-\delta s (1-s)\left\|\eta(u,v)\right\|_{v}^{m}.
         \end{eqnarray*}
        \end{proof}
    \begin{theorem}
    	Let $M$ be a complete RM, $S\subseteq M$ be a geodesic invex set w.r.t. $\eta$ and $F\colon S\times S\rightarrow R$ be a continuous strongly geodesic preinvex function of order $m$ w.r.t. $(\eta,\eta)$, $i.e.,$ $F$ is strongly geodesic preinvex function of order $m$ to each variable. Then, the function $\Psi\colon S\rightarrow R$ defined by \begin{eqnarray*} 
    		\Psi(u)=\inf\limits_{v\in S}F(u,v),
    	\end{eqnarray*} is strongly geodesic preinvex function of order $m$  w.r.t. $\eta$.
    \end{theorem}
\begin{proof}
	Let $u_{0},u_{1}\in S$ and $\epsilon>0$ is given. Since $S$ is geodesic invex set w.r.t. $\eta$, then there exist a geodesic $r_{u_{0},u_{1}}\colon [0,1]\rightarrow M$ s.t. 
	\begin{eqnarray*}
		r_{u_{0},u_{1}}(0)=u_{1},~r_{u_{0},u_{1}}{'}(0)=\eta(u_{0},u_{1}),~r_{u_{0},u_{1}}(s)\in S,~\forall~s\in[0,1].
	\end{eqnarray*}
By Definition of infimum, $\exists$ $v_{0},v_{1}\in S$ s.t.
\begin{eqnarray*}
	F(u_{1},v_{1})<\Psi(u_{1})+\epsilon, 	F(u_{0},v_{0})<\Psi(u_{0})+\epsilon.
\end{eqnarray*}
By geodesic invexity of $S$ w.r.t. $\eta$, $\exists$ a geodesic $t_{v_{0},v_{1}}\colon [0,1]\rightarrow M$ s.t. 
\begin{eqnarray*}
	t_{v_{0},v_{1}}(0)=v_{1},~t_{v_{0},v_{1}}{'}(0)=\eta(v_{0},v_{1}),~t_{v_{0},v_{1}}(s)\in S,~\forall~s\in[0,1].
\end{eqnarray*}
Clearly, the curve $\alpha_{(u_{0},v_{0}),(u_{1},v_{1})}=(r_{u_{0},u_{1}},t_{u_{0},u_{1}})\colon [0,1]\rightarrow M\times M$ is geodesic in $S\times S$ with $\alpha_{(u_{0},v_{0}),(u_{1},v_{1})}(0)=(u_{1},v_{1})$ s.t. for every $s\in[0,1]$, we have 
\begin{eqnarray*}
	\alpha_{(u_{0},v_{0}),(u_{1},v_{1})}(s)=\left(r_{u_{0},u_{1}}(s),t_{v_{0},v_{1}}(s)\right)\in S\times S
\end{eqnarray*}
 and \vspace*{-.9cm} 
\begin{eqnarray*}
	\alpha_{(u_{0},u_{1}),(v_{0},v_{1})}{'}\left(0\right)&=&\left(r_{u_{0},u_{1}}{'}(0),t_{v_{0},v_{1}}{'}(0)\right)\\& =&\left(\eta(u_{0},u_{1}),\eta(v_{0},v_{1})\right)\\&=&\eta_{0}\left((u_{0},v_{0}),(u_{1},v_{1})\right), 
\end{eqnarray*}
where the map $\eta_{0}\colon (M\times M)\times( M\times M)\rightarrow TM\times TM$.
By Definition of infimum and the strongly geodesic preinvexity of order $m$ of $F$, we have\\

$\Psi\left(r_{u_{0},u_{1}}(s)\right)$
\begin{eqnarray*}
\hspace{.9cm} &=&\inf\limits_{v\in S} F\left(r_{u_{0},u_{1}}(s),v\right)\\&\leq& F\left(r_{u_{0},u_{1}}(s),t_{v_{0},v_{1}}(s)\right)\\&\leq& s F\left(u_{0},v_{0}\right)+(1-s) F\left(u_{1},v_{1}\right)-\delta s(1-s)\left\|\eta_{0}\left((u_{0},v_{0}),(u_{1},v_{1})\right)\right\|^{m}_{(u_{1},v_{1})}\\&=& s F\left(u_{0},v_{0}\right)+(1-s) F\left(u_{1},v_{1}\right)-\delta s(1-s)\left\|\left(\eta(u_{0},u_{1}),\eta(v_{0},v_{1})\right)\right\|^{m}_{(u_{1},v_{1})}\\&=& s F\left(u_{0},v_{0}\right)+(1-s) F\left(u_{1},v_{1}\right)-\delta s(1-s)\left \{\left\|\eta(u_{0},u_{1})\right\|^{m}_{u_{1}}+\left\|\eta(v_{0},v_{1})\right\|^{m}_{v_{1}}\right\}\\&\leq&
	s\left(\Psi(u_{0})+\epsilon\right)+(1-s)\left(\Psi(u_{1})+\epsilon\right)-\delta s(1-s) \left\|\eta(u_{0},u_{1})\right\|^{m}_{u_{1}}
	\\&=&
	\left(s\Psi(u_{0})+(1-s)\Psi(u_{1})\right)+\epsilon-\delta  s(1-s) \left\|\eta(u_{0},u_{1})\right\|^{m}_{u_{1}}
		\\&\leq&
	s\Psi(u_{0})+(1-s)\Psi(u_{1})-\delta s(1-s) \left\|\eta(u_{0},u_{1})\right\|^{m}_{u_{1}}.
\end{eqnarray*}
 
Therefore, the function $\Psi(u)=\inf\limits_{v\in S}F(u,v)$ is strongly geodesic preinvex of order $m$ w.r.t. $\eta$.
\end{proof} 
        Now, we define strongly   invex function of order $m$  w.r.t. $\eta$. 
        \begin{definition}\label{9D}
        Let $m $ be a positive integer.	Let $M$ be a RM. A  differentiable function $h:M\rightarrow R$ is called strongly $\eta$-invex of order $m$ w.r.t. $\eta$, if $\exists~~\delta>0$ s.t.
        \begin{eqnarray*}
        	 h(u)\geq h(v)+dh_{v}^{T}(\eta(u,v))+\delta\|\eta(u,v)\|_{v}^{m},~~ \forall u,v\in M.
        \end{eqnarray*}
        \end{definition}
        \begin{remark}
        	 For m=2, then Definition \ref*{9D} reduce to strongly $\eta$-invex of order $2$ w.r.t. $\eta$ defined by \cite{Poury2}. 
         \end{remark}
           In the following example, we show the existence of strongly $\eta$-invex function of order $m$. 
        	 \begin{example}
        	 	Let  $h:M\rightarrow R$ be a differentiable function on RM $M$ s.t. $\forall v\in M$,~ $dh_{v}\neq 0$. A map $\eta:M\times M\rightarrow TM$ is defined by
        	 \begin{eqnarray*}
        	 	\hspace{2.8cm}	\eta(u,v)=\frac{h(u)-h(v)-\delta \|\eta(u,v)\|^{m}_{v}}{\|dh_{v}\|^{2}_{v}}dh_{v},\forall u,v\in M.
        	 \end{eqnarray*}
        	 		Since $M$ is RM, then
        	 	for every $u,v\in M$, we get
        	 \begin{eqnarray*}
        	 	 \left\langle dh_{v},\eta(u,v)\right\rangle_{v}&=&\left\langle dh_{v},\frac{h(u)-h(v)-\delta \|\eta(u,v)\|^{m}_{v}}{\|dh_{v}\|^{2}_{v}}dh_{v} \right\rangle_{v}\\
        	 	  &=&h(u)-h(v)-\delta \|\eta(u,v)\|^{m}_{v}\frac{\left\langle dh_{v},dh_{v}\right\rangle_{v}}{\|dh_{v}\|^{2}},\\
        	 	   &=&h(u)-h(v)-\delta \|\eta(u,v)\|^{m}_{v}.
        	 \end{eqnarray*}
        	\begin{eqnarray*}
        	\hspace{1.8cm}	h(u)=h(v)+\left\langle dh_{v},\eta(u,v)\right\rangle_{v}+\delta \|\eta(u,v)\|^{m}_{v}.
        	\end{eqnarray*}
        	 	\noindent
        	 	Hence, $h$ is strongly $\eta$-invex function of order $m$.  
        	 \end{example} 
        	 \begin{example}\label{Exm2}
        	 	Let $M=\left\{(u_{1},u_{2})\in R^{2}: u_{1},u_{2}>0\right\}$ be a RM with Riemannian metric $\langle\cdot,\cdot\rangle=<Q(u)x, y>$ for any $x,y \in T_u(M)$, where $Q(u) = (g_{ij}(u))$ defines a $2 \times 2$ matrix given by $g_{ij} (u)=\frac{\delta_{ij}}{u_iu_j}$.
        	 	The geodesic curve $r:\mathbb{R} \rightarrow M$ satisfying $r(0)=u=(u_1,u_2) \in M$ and $r'(0) = x=(x_1,x_2)\in T_u(M)=\mathbb{R}^2$ is given by
        	 	$$r(t) = (u_1e^{\frac{x_1}{u_1}t},~u_2e^{\frac{x_2}{u_2}t}).$$
        	 	For any $u=(u_1,u_2)\in M$ and any $x =(x_1,x_2) \in T_u(M)$, the map $exp_u: T_u(M) \rightarrow M$ is given by
        	 	$$exp_u(x) = r(1)=(u_1e^{\frac{x_1}{u_1}},~u_2e^{\frac{x_2}{u_2}})$$
     Define a function $h:M\rightarrow R$, for any $u=(u_{1},u_{2})\in M$, by $$h(u_{1},u_{2})=u_{1}+u_{2}^{2}.$$
     Then, $h$ is strongly $\eta$-invex of any order w.r.t $\eta$ (see Definition \ref{9D}), where $\eta:M\times M\rightarrow TM$ is a map defined, for any $u=(u_1,u_2), v=(v_1,v_2)\in M$, by
        	 	\begin{eqnarray*}
        	 		\eta\left((u_{1},u_{2}),(v_{1},v_{2})\right)=\left(-1-v_{1},-v_{2}\right),
        	 	\end{eqnarray*}
         	\begin{eqnarray*}
         	\text{and} \hspace{0.5cm}	dh_{v}^{T}(\eta(u,v))=\dfrac{d}{dt}h(\exp_{v}(\eta(u,v)))|_{t=0}=-(1 + v_1+2v_2^2).
         	\end{eqnarray*}
However, $h$ is not  strongly geodesic preinvex function of any order. For this, let $u=\left(\frac{1}{4},\frac{1}{4}\right)$~, $v= \left(\frac{1}{9},\frac{1}{9}\right)$ and $s=\frac{1}{10}$, then for any $\delta >0$ and any $m\geq 1$, it is easy to see the following holds:
\begin{eqnarray*}
	h(r_{u,v}(s))\nleq  s h(u)+(1-s)h(v)-\delta s (1-s)\|\eta(u,v)\|_{v}^{m},
\end{eqnarray*}
 where the geodesic segment joining $u=(u_{1},u_{2})$ and $v=(v_{1},v_{2})$ is given by
\begin{eqnarray*}
	r_{(u,v)}(s)=\left(u_{1}\left(\frac{v_{1}}{u_{1}}\right)^{s},u_{2}\left(\frac{v_{2}}{u_{2}}\right)^{s}\right).
\end{eqnarray*}  
        	 \end{example} 
      In the following theorem, we show differentiable strongly geodesic preinvex function is strongly $\eta$-invex.
        \begin{theorem}\label{2}
        Let $S\subseteq M$ be an open geodesic invex set and $h:S\rightarrow R$ be continuously differentiable function. If   $h$ is strongly geodesic preinvex of order $m$ w.r.t. $\eta$, then $h$ is strongly $\eta$-invex of order $m$ w.r.t. $\eta$.
        \end{theorem} 
        \begin{proof}
        Assume $h$ is strongly geodesic preinvex of order $m$ w.r.t. $\eta$, for every $u,v\in S$, $\exists$ exactly one geodesic $r_{u,v}:[0,1]\rightarrow M$ s.t.
        \begin{eqnarray*}
        	r_{u,v}(0)=v,~r_{u,v}{'}(0)=\eta(u,v),~r_{u,v}(s)\in S,~\forall~ s\in[0,1],
        \end{eqnarray*}
        	and \vspace*{-.3cm}
        \begin{eqnarray*}
        	h(r_{u,v}(s))\leq  s h(u)+(1-s)h(v)-\delta s (1-s)\|\eta(u,v)\|_{v}^{m},
        \end{eqnarray*}
        \begin{eqnarray*}
        	h(r_{u,v}(s))-h(v)\leq  s \{h(u)-h(v)\}-\delta s (1-s)\|\eta(u,v)\|_{v}^{m}.
        \end{eqnarray*}
        	Since $h$ is differentiable,  dividing by $s$ on both side and taking   $s\rightarrow 0$, we get
        \begin{eqnarray*}
        	 dh_{r_{u,v}(0)}^{T}(r_{u,v}{'}(0))\leq h(u)-h(v)-\delta\|\eta(u,v)\|_{v}^{m}.
        \end{eqnarray*}
        	or \vspace*{-.3cm}
         \begin{eqnarray*}
         	 h(u)\geq h(v)+dh_{v}^{T}(\eta(u,v))+\delta\|\eta(u,v)\|_{v}^{m}.
         \end{eqnarray*}
        \end{proof} 
       However, the converse of  Theorem \ref{2} holds when $\eta$ satisfies {\bf{Condition C}} as follows:  
       \begin{theorem} 
       	Let $S\subseteq M$ be an open geodesic invex set w.r.t. $\eta$. Assume $h:S\rightarrow R$ is continuously differentiable function. The function $h$ is strongly $\eta$-invex of order $m$ w.r.t. $\eta$  and $\eta$ satisfies {\bf{Condition C}} if and only if $h$ is strongly geodesic preinvex of order $m$ w.r.t. $\eta$.
       \end{theorem}
       \begin{proof} By Theorem \ref{2}, then the function $h$ is strongly $\eta$-invex of order $m$ w.r.t. $\eta$.\\
       	 Conversely, suppose $h$ is strongly $\eta$-invex of order $m$ w.r.t. $\eta$ on open geodesic invex set $S$ w.r.t. $\eta$  $i.e.,$ for every $u,v\in S~\exists$ a exactly one curve $r_{u,v}:(0,1)\rightarrow M$ s.t.
       	\begin{eqnarray*}
       		r_{u,v}(0)=v,~r_{u,v}{'}(0)=\eta(u,v),~r_{u,v}(s)\in S,~\forall~s\in(0,1).
       	\end{eqnarray*}
       	 Fixed $s\in(0,1)$ and setting $\bar{v}=r_{u,v}(s).$ Then, we have
       	 \begin{eqnarray}\label{3.1}
       	   h(u)\geq h(\bar{v})+dh_{\bar{v}}^{T}(\eta(u,\bar{v}))+\delta\|\eta(u,\bar{v})\|^{m} 
       	 \end{eqnarray}
       	 \begin{eqnarray}\label{3.2}
       	  h(v)\geq h(\bar{v})+dh_{\bar{v}}^{T}(\eta(v,\bar{v}))+\delta\|\eta(v,\bar{v})\|^{m}.
       	 \end{eqnarray}
       	  By multiplying $s$ in  (\ref{3.1}) and $(1-s)$ in   (\ref{3.2}) respectively, adding and applying\\ 
       	  $dh_{\bar{v}}^{T}\left( s\eta(u,\bar{v})+(1-s) \eta(v,\bar{v})\right)$
       	  \begin{eqnarray*}
       	  \hspace{2cm}	&=&dh_{\bar{v}}\left(s(1-s)P^{s}_{0,r_{u,v}}\left[\eta(u,v)\right]+(1-s)(-s)P^{s}_{0,r_{u,v}}\left[\eta(u,v)\right]\right)\\&=&dh_{\bar{v}}(0)\\&=&0,
       	  \end{eqnarray*}
         we have \vspace*{-.2cm}
     \begin{eqnarray*}
     	h(\bar{v})+\delta s  \left\|(1-s)P^{s}_{0,r_{u,v}}\left[\eta(u,v)\right]\right\|_{v}^{m} +\delta(1-s)  \left\|(-s)P^{s}_{0,r_{u,v}}\left[\eta(u,v)\right]\right\|_{v}^{m} 
     \end{eqnarray*}
    \begin{eqnarray*}
  \hspace{10cm}\leq sh(u)+(1-s) h(v)
    \end{eqnarray*}
$h(\bar{v})+\delta s  (1-s)^{m}\left\|P^{s}_{0,r_{u,v}} \eta(u,v)\right\|_{v}^{m}+\delta(1-s)s^{m}\left\|P^{s}_{0,r_{u,v}} \eta(u,v)\right\|_{v}^{m}$
    \begin{eqnarray*}
   \hspace{10cm} \leq sh(u)+(1-s) h(v)	
    \end{eqnarray*}
          \begin{equation}\label{3.3}
       	  h(\bar{v})+\delta s (1-s)\left[(1-s)^{m-1}+s^{m-1}\right]\left\|P^{s}_{0,r_{u,v}} \eta(u,v)\right\|_{v}^{m}\leq sh(u)+(1-s) h(v).
       	  \end{equation}
         Since $P^{s}_{0,r_{u,v}} \eta(u,v)= \eta(u,v)$, then (\ref{3.3})  become
          \begin{equation}\label{3.4}
         	h(\bar{v})+\delta s (1-s)\left[(1-s)^{m-1}+s^{m-1}\right]\left\|  \eta(u,v)\right\|_{v}^{m}\leq sh(u)+(1-s) h(v).
         \end{equation}
       	  {\bf Case (i)} $0<m\leq2$, then $ (1-s)+s=1\leq (1-s)^{m-1}+s^{m-1}$.\\ 
       	   {\bf Case (ii)} $m>2$, then the real valued function $\phi(s)=s^{m-1}$ is convex on $(0,1)$, thus we have
       	  \begin{eqnarray*}
       	  	\left(\frac{1}{2}\right)^{m-2}\leq(1-s)^{m-1}+s^{m-1}.
       	  \end{eqnarray*}
       	   It follows that  (\ref{3.4}), $\exists~\delta^{'} >0$, which is independent from $u,v,s$ such that 
       	  \begin{eqnarray*} 
       	  	h(r_{u,v}(s))\leq  s h(u)+(1-s)h(v)-\delta^{'} s (1-s) \left\|\eta(u,v)\right\|_{v}^{m},~\forall u,v\in S,~s \in(0,1).
       	  \end{eqnarray*} 
       	   Hence, the function $h$ is strongly geodesic preinvex function of order $m$.
       \end{proof} 
         Now, we generalize strongly $\eta$-invex function of order $m$ as follows$\colon$ 
         \begin{definition}\label{11}
         Let $m $ be a positive integer. Let $ h:M\rightarrow R$ be a differentiable function on RM $M$.
          Then, the function $h$ is called$\colon$\\ 
          {\bf (i)} Strongly pseudo $\eta$-invex type $1$ of order $m$, if  $\exists ~~\delta>0$ s.t.
         	\begin{eqnarray*}
         		dh_{v}^{T}\eta(u,v)\geq 0 \implies h(u)\geq h(v)+\delta \|\eta(u,v)\|^{m}_{v}, ~\forall u,v\in M,
         	\end{eqnarray*}
         		or equivalently, 
         		\begin{eqnarray*}
         			h(u)<h(v)+\delta \|\eta(u,v)\|^{m}_{v}   \implies dh_{v}^{T}\eta(u,v)< 0.
         		\end{eqnarray*} 
         		{\bf (ii)}  Strongly pseudo $\eta$-invex type $2$ of order $m$, if  $\exists ~~\delta>0$ s.t.
         	\begin{eqnarray*}
         		dh_{v}^{T}\eta(u,v)+\delta \|\eta(u,v)\|^{m}_{v}\geq 0 \implies h(u)\geq h(v) , ~\forall u,v\in M.
         	\end{eqnarray*}
         \end{definition}
     In the following example, we show the existence of  strongly pseudo $\eta$-invex function of order $m$.
 \begin{example}
 Let $M=\left\{(u_{1},u_{2})\in R^{2}: u_{1},u_{2}>0\right\}$ be a RM with Riemannian metric as defined in Example \ref{Exm2}.\\ 
 Define a function $h:M\rightarrow R$, for any $u=(u_{1},u_{2})\in M$, by $$h(u_{1},u_{2})=ln(u_{1})+(ln(u_{2}))^{3}.$$
 Then, $h$ is strongly pseudo $\eta$-invex type $1$ of any order w.r.t $\eta$, where $\eta:M\times M\rightarrow TM$ is a map defined, for any $u=(u_1,u_2), v=(v_1,v_2)\in M$, by
 \begin{eqnarray*}
 	\eta\left((u_{1},u_{2}),(v_{1},v_{2})\right)=\left(-v_{1}^{2},0\right),
 \end{eqnarray*}
 \begin{eqnarray*}
 	\text{and} \hspace{0.5cm}	dh_{v}^{T}(\eta(u,v))=\dfrac{d}{dt}h(\exp_{v}(\eta(u,v)))|_{t=0}=-v_1.
 \end{eqnarray*}
 However, $h$ is not strongly $\eta$-invex function of any order. For this, let $u= \left(\frac{1}{2},\frac{1}{2}\right)$~, $v= \left(1,e^{2}\right)$, then it is easy to see the following
 \begin{eqnarray*}
	h(u)\ngeq h(v)+dh_{v}^{T}\eta(u,v)+\delta\|\eta(u,v)\|_{v}^{m}.
\end{eqnarray*} 
 \end{example}
     \begin{definition}\label{13}
     Let $m $ be a positive integer. Let $ h:M\rightarrow R$ be a differentiable function on RM $M$.Then, the function $h$ is called$\colon$\\ 
     	{\bf (i)} Strongly quasi $\eta$-invex type $1$ order of $m$, if  $\exists ~~\delta>0$ s.t.
     \begin{eqnarray*}
     	h(u)\leq h(v) \implies dh_{v}^{T}\eta(u,v)+\delta \|\eta(u,v)\|^{m}_{v} \leq 0, ~~~\forall~~u,v\in M.
     \end{eqnarray*} 
 {\bf (ii)} Strongly quasi $\eta$-invex type $2$ order of $m$, if  $\exists ~\delta>0$ s.t.
 \begin{eqnarray*}
 	h(u)\leq h(v)+\delta \|\eta(u,v)\|^{m}_{v} \implies dh_{v}^{T}\eta(u,v) \leq 0, ~~~\forall~~u,v\in M.
 \end{eqnarray*}
     \end{definition}
      In the following example, we show the existence of   strongly quasi $\eta$-invex function of order $m$.
        \begin{example}
        	Let $M=\left\{(u_{1},u_{2})\in R^{2}: u_{1},u_{2}>0\right\}$ be a RM with Riemannian metric as defined in Example \ref{Exm2}.\\ 
        	Define a function $h:M\rightarrow R$, for any $u=(u_{1},u_{2})\in M$, by $$h(u_{1},u_{2})=u_{1}^{3}+ln(u_{2}).$$
        	Then, $h$ is strongly quasi $\eta$-invex type $1$ of any order w.r.t $\eta$, where $\eta:M\times M\rightarrow TM$ is a map defined, for any $u=(u_1,u_2), v=(v_1,v_2)\in M$, by
        	\begin{eqnarray*}
        		\eta\left((u_{1},u_{2}),(v_{1},v_{2})\right)=\left(-v_{1}^{2},-v_{2}^{2}\right),
        	\end{eqnarray*}
        	\begin{eqnarray*}
        		\text{and} \hspace{0.5cm}	dh_{v}^{T}(\eta(u,v))=\dfrac{d}{dt}h(\exp_{v}(\eta(u,v)))|_{t=0}=-3v_1^{4}-v_{2}\leq0.
        	\end{eqnarray*}
        	However, $h$ is not strongly $\eta$-invex function of any order. For this, let $u= \left(1,\frac{1}{e^{5}}\right)$~, $v= \left(1,1\right)$, then it is easy to see the following
        	\begin{eqnarray*}
        		h(u)\ngeq h(v)+dh_{v}^{T}\eta(u,v)+\delta\|\eta(u,v)\|_{v}^{m}.
        	\end{eqnarray*}
        \end{example}
         \section{\bf{Strongly invariant $\eta$-monotone on Riemannian manifolds}}\label{4a}
             The monotonicity of vector field on RM defined by Nemeth \cite{Nemeth} as follows$\colon$ 	 
         \begin{definition}\cite{Nemeth}
         	Let   $X$ be a vector field on RM $M$. Then, $X$ is called  monotone on $M$ if $\forall~u,v\in M$, we get 
         	\begin{eqnarray*}
         		\left\langle r_{u,v}{'}(0),~~~P^{0}_{1,  r_{u,v}}\left[X(u)-X(v)\right]\right\rangle_{v}\geq 0,
         	\end{eqnarray*}
         	where $r_{u,v}$ is a geodesic joining $u$ and $v$.
         \end{definition} 
           Barani et al. \cite{Poury2} generalized it and defined invariant monotonicity on RM $M$. Later Iqbal et al. \cite{Akhlad1} extended the notion of invariant monotonicity to strongly invariant $\alpha$-monotonicity on RM $M$.  Motivated by Iqbal et al. \cite{Akhlad1}, we extend it as follows:
         \begin{definition} \label{Def 15}
         	 Let $m $ be a positive integer. Let $M$ be a RM. A vector field $X$ on  $M$ is called$\colon$\\ 
         	  {{\bf (i)}} strongly invariant $\eta$-monotone of order $m$ if $\exists$   $\delta>0$ s.t.
         \begin{eqnarray*}
         	\left\langle X(v),\eta(u,v)\right\rangle_{v}+\langle X(u),\eta(v,u)\rangle_{u}\leq-\delta\left\{\|\eta(u,v)\|^{m}_{v}+\|\eta(v,u)\|^{m}_{u}\right\}.
         \end{eqnarray*}
     For $m=2$, it reduces to strongly invariant monotone defined by Barani et al. \cite{Poury2}.
       {{\bf (ii)}} strongly invariant pseudo $\eta$-monotone of order $m$ if $\exists~\delta>0$ s.t.
     	\begin{eqnarray*}
     		\left\langle X(u),\eta(v,u)\right\rangle_{u}\geq 0\implies\left\langle X(v),\eta(u,v)\right\rangle_{v}\leq-\delta \left\|\eta(u,v)\right\|^{m}_{v}.
     	\end{eqnarray*}
     	For $m=2$, it reduces to strongly invariant pseudo monotone defined by Barani et al. \cite{Poury2}.
         \end{definition}
         In the support of our Definition \ref*{Def 15}, we give the following example.
         \begin{example}
         	Let $ h:M\rightarrow R$ be a differentiable function on RM $M$ s.t. $\forall ~v\in M$,~~~ $dh_{v}\neq 0$. A map $\eta:M\times M\rightarrow TM$ is defined by
         \begin{eqnarray*}
         \hspace{2.3cm}	\eta(u,v)=-\frac{  \|\eta(u,v)\|^{m}_{v}}{\|dh_{v}\|^{2}_{v}}dh_{v},~~~\forall u,v\in M.
         \end{eqnarray*}
         Given $M$ is RM, then
         	for every $u,v\in M$, we have
         \begin{eqnarray*}
         	\left\langle dh_{v},\eta(u,v)\right\rangle_{v}&=&\left\langle dh_{v},\frac{  -\|\eta(u,v)\|^{m}_{v}}{\|dh_{v}\|^{2}_{v}}dh_{v} \right\rangle_{v}\\
         	&=& -\|\eta(u,v)\|^{m}_{v}\frac{\left\langle dh_{v},dh_{v}\right\rangle_{v}}{\|dh_{v}\|^{2}},\\
         	&=&  -{\|\eta(u,v)\|^{m}_{v}},
         \end{eqnarray*} 
         	\begin{equation}\label{5.1a}
         		\langle dh_{v},\eta(u,v)\rangle_{v}=-{\|\eta(u,v)\|^{m}_{v}}.\hspace{2.3cm}
         	\end{equation} 
         	Similarly, we have 
         	\begin{equation}\label{5.2}
         		\langle dh_{u}\eta(v,u)\rangle_{u}=-{\|\eta(v,u)||^{m}_{u}}.
         	\end{equation} 
         	By adding  (\ref{5.1a}) and (\ref{5.2}), we get
         \begin{eqnarray*}
         	 \langle dh_{v},\eta(u,v)\rangle_{v}+\langle dh_{u},\eta(v,u)\rangle_{u}=-\left\{\|\eta(u,v)\|^{m}_{v}+\|\eta(v,u)\|^{m}_{u}\right\},
         \end{eqnarray*}
         Hence, $dh_{v}$ is strongly invariant  $\eta$-monotone vector field of order $m$.	
         \end{example} 
     
     In the next Theorems, we discuss a relationship between strongly $\eta$-invex of order $m$ and strongly invariant $\eta$-monotone of order $m$.
         \begin{theorem}
         	Let $ h:M\rightarrow R$ be a differentiable function on   $M$. Suppose $h$ is strongly $\eta$-invex of order $m$. Then, $dh$ is strongly invariant $\eta$-monotone of order $m$.
         \end{theorem}
         \begin{proof}
         	Let $h$ be a strongly $\eta$-invex of order $m$ on $M$. Then,
         	\begin{equation}\label{5.3}
         	h(u)- h(v)\geq   dh_{v}(\eta(u,v)) +\delta \|\eta(u,v)\|^{m}_{v},~	\forall~u,v\in M,	
         	\end{equation} 
         		\begin{equation}\label{5.4}
         		h(v)- h(u)\geq   dh_{u}(\eta(v,u)) +\delta \|\eta(v,u)\|^{m}_{u},~\forall~u,v\in M.	
         		\end{equation}
         	 Adding  (\ref{5.3}) and (\ref{5.4}), we get 
         \begin{eqnarray*}
         	 dh_{v}(\eta(u,v))+dh_{u}(\eta(v,u))+\delta\{\|\eta(u,v)\|^{m}_{v}+\|\eta(v,u)\|^{m}_{u}\} \leq0,
         \end{eqnarray*}
         	 or
         	\begin{eqnarray*}
         		dh_{v}(\eta(u,v))+dh_{u}(\eta(v,u))\leq-\delta\{\|\eta(u,v)\|^{m}_{v}+\|\eta(v,u)\|^{m}_{u}\},~~ \forall u,v\in M.
         	\end{eqnarray*}
         	Thus, $dh$ is strongly invariant $\eta$-monotone of order $m$. 
         \end{proof} 
         \begin{theorem}
         	Let $h:M\rightarrow R$ be a differentiable function on geodesically  complete RM $M$. If a map $\eta:M\times M\rightarrow TM$ is integrable and $dh$ is strongly invariant $\eta$-monotone of order $m$. Then, $h$ is strongly $\eta$-invex of order $m$.
         \end{theorem} 
         \begin{proof}
        Since $M$ is a geodesically complete RM,   $\forall~u,v\in M$, then there exists a geodesic $r_{u,v}:[0,1]\rightarrow M$ s.t.	$r_{u,v}(0)=v,~~~r_{u,v}(1)=u$. Let $w=r_{u,v}(\frac{1}{2})$, then by the mean value theorem (MVT), $\exists ~~s_{1},s_{2}\in (0,1)$ s.t. $0<s_{2}<\frac{1}{2}<s_{1}<1$, we get
         \begin{equation}\label{5.5}
         h(u)-h(w)=\frac{1}{2}dh_{x}\left(r_{u,v}{'}(s_{1})\right),
         \end{equation} 
         \begin{equation} \label{5.6}
         h(w)-h(v)=\frac{1}{2}dh_{y}\left(r_{u,v}{'}(s_{2})\right),
         \end{equation}
         where $x=r_{u,v}(s_{1}),~~~y=r_{u,v}(s_{2})$. 
         Since $dh$ is strongly invariant $\eta$-monotone of order $m$, then we get
         \begin{equation}\label{5.7}
          dh_{x}^{T}\left(\eta(v,x)\right)+dh_{v}^{T}\left(\eta(x,v)\right)\leq-\delta\left\{\|\eta(v,x)\|^{m}_{x}+\|\eta(x,v)\|^{m}_{v}\right\}.
         \end{equation} 
          Since $\eta$ is integrable, we get
          \begin{equation}\label{5.8}
           \eta\left(v,r_{u,v}(s_{1})\right)=-s_{1}P_{0,r_{u,v}}^{s_{1}}\left[\eta(u,v)\right]
          \end{equation}
          and
           \begin{equation}\label{5.9}
           \eta\left(v,r_{u,v}(s_{1})\right)=s_{1} \eta(u,v).
           \end{equation} 
           Using (\ref{5.8}) and  (\ref{5.9}) in (\ref{5.7}), we get\\
           
           $dh_{x}^{T}\left(-s_{1}P_{0,r_{u,v}}^{s_{1}}\left[\eta(u,v)\right]\right)+dh_{v}^{T}\left(s_{1} \eta(u,v)\right)$
          \begin{eqnarray*}
          	\hspace{5cm}\leq-\delta\left\{\|-s_{1}P_{0,r_{u,v}}^{s_{1}}[\eta(u,v)]\|^{m}_{x}+\|s_{1} \eta(u,v)\|^{m}_{v}\right\}
          \end{eqnarray*}
           or
          \begin{eqnarray*}
          	-dh_{x}^{T}\left(P_{0,r_{u,v}}^{s_{1}}\eta(u,v)\right)+dh_{v}^{T}\left( \eta(u,v)\right)\leq-2 s_{1}^{m-1}\delta \left\| \eta(u,v)\right\|^{m}_{v}
          \end{eqnarray*}
          From {\bf{Condition~C}},
          $P_{0,r_{u,v}}^{s_{1}}\left[\eta(u,v)\right]=r_{u,v}{'}({s_{1}})$, then we have 
          \begin{equation}\label{5.10}
           \frac{1}{2}dh_{x}^{T}\left(r_{u,v}{'}({s_{1}})\right)\geq\frac{1}{2}dh_{v}^{T}\left(\eta(u,v)\right)+ s_{1}^{m-1}\delta \|\eta(u,v)\|^{m}_{v},
          \end{equation}
          similarly, we have
          \begin{equation}\label{5.11}
          \frac{1}{2}dh_{y}^{T}\left(r_{u,v}{'}({s_{2}})\right)\geq\frac{1}{2}dh_{v}^{T}\left(\eta(u,v)\right)+ s_{2}^{m-1}\delta \|\eta(u,v)\|^{m}_{v}
          \end{equation}
       Hence, (\ref{5.5}) and (\ref{5.6}) become
        \begin{equation}\label{5.12} 
        h(u)-h(w)\geq\frac{1}{2}dh_{v}^{T}(\eta(u,v))+ s_{1}^{m-1}\delta \| \eta(u,v)\|^{m}_{v},
        \end{equation} 
        \begin{equation}\label{5.13}  
        h(w)-h(v)\geq\frac{1}{2}dh_{v}^{T}(\eta(u,v))+ s_{2}^{m-1}\delta \| \eta(u,v)\|^{m}_{v}.
        \end{equation} 
        
        \noindent
        Adding (\ref{5.12}) and (\ref{5.12}), we get
        \begin{eqnarray*}
        	h(u)-h(v)&\geq& dh_{v}^{T}( \eta(u,v))+ \delta\left\{s_{1}^{m-1}+s_{2}^{m-1}\right\} \left\| \eta(u,v)\right\|^{m}_{v}\\&=&dh_{v}^{T}\left( \eta(u,v)\right)+ \delta^{'} \left\| \eta(u,v)\right\|^{m}_{v},
        \end{eqnarray*}
         where $\delta^{'}=\delta\left(s_{1}^{m-1}+s_{2}^{m-1}\right)>0$. 
         \end{proof}
          \begin{theorem}
          	Let $M$ be a geodesically  complete RM and $~h:M\rightarrow R$ be a differentiable function. If a map $\eta:M\times M\rightarrow TM$ is integrable and $dh$ is strongly invariant pseudo $\eta$-monotone of order $m$ . Then,   $h$ is strongly pseudo $\eta$-invex of order $m$.
          \end{theorem}
          \begin{proof}
         Since $M$ is  geodesically complete RM, then $\forall~~u,v\in M$, $\exists$  a geodesic $r_{u,v}:[0,1]\rightarrow M$ s.t.	$r_{u,v}(0)=v,~r_{u,v}(1)=u$. Let $w=r_{u,v}(\frac{1}{2})$, then from (MVT),    $\exists~~~s_{1},s_{2}\in (0,1)$ s.t. $0<s_{2}<\frac{1}{2}<s_{1}<1$, we get
          	\begin{equation}\label{5.14}
          	h(u)-h(w)=\frac{1}{2}dh_{x}^{T}\left(r_{u,v}{'}(s_{1})\right),
          	\end{equation} 
          	
          	\begin{equation} \label{5.15}
          	h(w)-h(v)=\frac{1}{2}dh_{y}^{T}\left(r_{u,v}{'}(s_{2})\right),
          	\end{equation}
          	where $x=r_{u,v}(s_{1}),~~~y=r_{u,v}(s_{2})$.
          Using the property (P), (\ref{5.14})can be written as  
          	
          	\begin{equation}\label{5.16}
          	h(u)-h(w)=-\frac{1}{2s_{1}}dh_{x}^{T}\left(\eta(v,x)\right)
          	\end{equation} 
          	and \vspace*{-.4cm}
          	\begin{equation} \label{5.17}
          	h(w)-h(v)=-\frac{1}{2s_{2}}dh_{y}^{T}\left(\eta(v,y)\right).
          	\end{equation}
          	Assume that \vspace*{-.4cm}
          	\begin{equation}\label{5.18}
          	 dh_{v}^{T}\left(\eta(u,v)\right)\geq 0,
          	\end{equation}
          	 using property $(P)$ in (\ref{5.18}), we get
          	 \begin{equation}\label{5.19}
          	 0\leq dh_{v}^{T}\left(\eta(u,v)\right)=\frac{1}{s_{1}}dh_{v}^{T}\left(\eta(v,x)\right)=\frac{1}{s_{2}}dh_{v}^{T}\left(\eta(v,y)\right). 
          	 \end{equation}
          Since $dh$ is strongly invariant pseudo $\eta$-monotone of order $m$, we have
          	\begin{equation}\label{3.24}
          		\begin{split}
          	dh_{x}^{T}\left(\eta(v,x)\right) &\leq-\delta\left\|\eta(v,x)\right\|^{m}_{x}\\&=-\delta\left\|-s_{1}P_{0,r_{u,v}}^{s_{1}}\left[\eta(u,v)\right]\right\|^{m}_{x}\\&=-\delta s_{1}^{m}\left\|P_{0,r_{u,v}}^{s_{1}}\left[\eta(u,v)\right]\right\|^{m}_{x}\\&=-\delta s_{1}^{m}\left\| \eta(u,v)\right\|^{m}_{v},
          	\end{split} 
          	\end{equation}
          and
          	\begin{equation}\label{3.25}
          	\begin{split}
          		dh_{y}^{T}\left(\eta(v,y)\right) &\leq-\delta\left\|\eta(v,y)\right\|^{m}_{y}\\&=-\delta\left\|-s_{2}P_{0,r_{u,v}}^{s_{2}}\left[\eta(u,v)\right]\right\|^{m}_{y}\\&=-\delta s_{2}^{m}\left\|P_{0,r_{u,v}}^{s_{2}}\left[\eta(u,v)\right]\right\|^{m}_{y}\\&=-\delta s_{2}^{m}\left\| \eta(u,v)\right\|^{m}_{v}.
          	\end{split} 
          \end{equation} 
          	Using (\ref{3.24}) and (\ref{3.25}) in   (\ref{5.16}) and (\ref{5.17}) respectively, we get
          	 \noindent
          	 \begin{equation}\label{5.23}
          	 \hspace{1cm} h(u)-h(w)\geq \dfrac{\delta s_{1}^{m-1}}{2}\left\|\eta(u,v)\right\|^{m}_{v}\vspace*{-.4cm}
          	 \end{equation}
          	where $\dfrac{\delta s_{1}^{m-1}}{2}>0$,\vspace*{-.5cm}
          	 
          	 \begin{equation}\label{5.24}
          	\hspace{1cm} h(w)-h(v)\geq \dfrac{\delta s_{2}^{m-1}}{2}\left\|\eta(u,v)\right\|^{m}_{v},
          	 \end{equation}
          	 	where $\dfrac{\delta s_{2}^{m-1}}{2}>0$.
          	 Adding (\ref{5.23}) and (\ref{5.24}), we get
          \begin{eqnarray*}
          	h(u)-h(v)\geq \delta^{'}\left\|\eta(u,v)\right\|^{m}_{v}
          \end{eqnarray*}
      or\vspace*{-.4cm}
      \begin{eqnarray*}
      	h(u)\geq h(v)+ \delta^{'}\left\|\eta(u,v)\right\|^{m}_{v},
      \end{eqnarray*}
          	  where $\delta^{'}=\dfrac{\delta\left(s_{1}^{m-1}+s_{2}^{m-1}\right)}{2}>0$.
          \end{proof}
       \section{\bf{Vector variational-like inequality problem on Riemannian manifolds}}\label{5a}
       In this section, we consider  the multi-objective optimization problem $(MOP)$ as an application for strongly $\eta$-invex functions of order $m$, known as strongly $\eta$-invex multi-objective optimization problem, which generalizes the results obtained by Iqbal et al. \cite{Akhlad1}.\\
       Suppose $H=(h_{1},h_{2},...,h_{k})$, where $h_{i}:M\rightarrow 2^{TM}$  are set valued vector fields on $M$.
       The vector variational-like inequality problem  $(VVLIP)$ is to find a solution $u^{*}\in M$, and $X\in H(u^{*})$ s.t.
     \begin{eqnarray*}
     	\left\langle X,\eta (u,u^{*})\right\rangle\nless 0 ,~~~\forall~~ u\in M,
     \end{eqnarray*}\vspace*{-.4cm}
       where \begin{eqnarray*}
       	\left\langle X,\eta(u,u^{*})\right\rangle=\langle X_{1},\eta (u,u^{*})\rangle,\langle X_{2},\eta(u,u^{*})\rangle,...,\langle X_{k},\eta(u,u^{*})\rangle. 
       \end{eqnarray*}
       The $MOP$ is to find a strict $\eta$-minimizers of order $m$ for$\colon$
       \begin{eqnarray*}
       	\mbox{Min}~  H(u)=\left(h_{1}(u),h_{2}(u),...,h_{k}(u)\right),~~~~ u\in M.
       \end{eqnarray*}
   
       Motivated by Iqbal et al. \cite{Akhlad1} and Bhatia et al. \cite{Bhatia}, we define a  local strict $\eta$-minimizers of order $m$ with respect to a nonlinear function on RM for $MOP$.
       \begin{definition} 
       	A point $u^{*}\in M $ is a local strict minimizers if     $\exists~\epsilon>0$ s.t. $h(u) \nless h(u^{*})$  $\forall~u\in B(u^{*},\epsilon)\cap M$,  $i.e.,$ there exist no $u\in B(u^{*},\epsilon)\cap M$, s.t. $h(u)< h(u^{*})$.
       \end{definition} 
       \begin{definition}
       	Let $m $ be a positive integer. A point $u^{*}\in M$ is a local strict $\eta$-minimizers of order $m$, if $\exists ~\epsilon>0$ and $\delta 
       	>0$ s.t.
       \begin{eqnarray*}
       	h(u) \nless h(u^{*})+\delta\|\eta(u,u^{*})\|^{m},~~\forall~u\in B(u^{*},\epsilon)\cap M.
       \end{eqnarray*}
        
       \end{definition}
       The local strict $\eta$-minimizers change to the strict $\eta$-minimizers if an open ball $B(u^{*},\epsilon)$ is replaced by  RM $M$.
       \begin{definition}\label{Defin 17}
       	Let $m $ be a positive integer. A point $u^{*}\in M$ is a strict $\eta$-minimizers of order $m$, if $\exists$ $\delta  >0$ s.t.
       \begin{eqnarray*}
       	h(u) \nless h(u^{*})+\delta\|\eta(u,u^{*})\|^{m},~\forall~~~u\in M.
       \end{eqnarray*}
       \end{definition} 
   
       In next Theorem, we show an important characterization of  a solution of $VVLIP$  and  a strict $\eta$-minimizers of order $m$ for $MOP$.
       \begin{theorem}\label{7T}
       	 Let $h_{i}, 1\leq i\leq k$, be a strongly $\eta$-invex functions of order $m$ on $M$. Then, $u^{*}\in M$ is a solution of $VVLIP\iff~u^{*}$ is a strict $\eta$-minimizers of order $m$ for $MOP$.
       \end{theorem}
       \begin{proof}
       	 Assume $u^{*}$ is a solution of $VVLIP$ but $u^{*}$ is not a strict $\eta$-minimizers of order $m$ for  $MOP$. Then, for all $\delta>0$,  $\exists$ some $\bar{u}\in M$, s.t.
       	 \vspace*{-.1cm}
       	 \begin{eqnarray*}
       	 	h(\bar{u})<h(u^{*})+\delta\|\eta(\bar{u},u^{*})\|^{m},\vspace*{-.2cm}
       	 \end{eqnarray*}
       	 $i.e.,$\vspace*{-.3cm}
       	 \begin{equation}\label{5.1}
       	  h_{i}(\bar{u})<h_{i}(u^{*})+\delta\|\eta(  \bar{u},u^{*})\|^{m}.	\hspace{.5cm}
       	 \end{equation}
       	 
       	 \noindent 
       	 Since $h_{i},~1\leq i\leq k$, are strongly  $\eta$-invex functions of order $m$, then equation (\ref{5.1}) yields
       	 \vspace*{-.4cm}
       	 \begin{eqnarray*}
       	 		\hspace{2cm}\langle X_{i},\eta( \bar{u},u^{*})\rangle< 0 ,~\forall X_{i}\in T_{i}(u^{*})=\nabla h_{i}(u^{*}) ,~~~1\leq i\leq k,
       	 \end{eqnarray*}   $i.e.,$ \vspace*{-.4cm}
       	  \begin{eqnarray*}
       	  	\langle X,\eta( \bar{u},u^{*})\rangle< 0,~\forall X \in T(u^{*}),\bar{u}\in M, 
       	  \end{eqnarray*}
          which contradicts that $u^{*}$ is a solution to $VVLIP$.\
       	  Hence, $u^{*}$ is a   strict $\eta$-minimizers of order $m$ for $MOP$.\\ 
       	  Conversely, assume $u^{*}$ is a strict $\eta$- minimizers of order $m$ for $MOP$, but $u^{*}$ is not a solution to $VVLIP$. Then, there exists $\bar{u}\in M$ s.t. 
       \begin{eqnarray*}
       	\hspace{2.5cm}\langle X_{i},\eta(\bar{u},u^{*})\rangle< 0 ,~\forall~ X_{i}\in T_{i}(u^{*})=\nabla h_{i}(u^{*}),~1\leq i\leq k.
       \end{eqnarray*}
         By Definition \ref{9D}, we get \vspace*{-.3cm}
       \begin{eqnarray*}
       \hspace{2cm}	h_{i}(\bar{u})<h_{i}(u^{*})+\delta\|\eta( \bar{u},u^{*})\|^{m},~1\leq i\leq k,
       \end{eqnarray*} 
       	 $i.e.,$\vspace*{-.4cm}
       	\begin{eqnarray*}
       		h(\bar{u})<h(u^{*})+\delta\|\eta(\bar{u},u^{*})\|^{m},
       	\end{eqnarray*}
       	  \noindent 
       	  which contradicts $u^{*}$ is a strict $\eta$-minimizers of order $m$. 
       	 Thus, $u^{*}$ is a solution to $VVLIP$.
       \end{proof}
   \section{\bf Conclusion}
   The strong concept of strongly geodesic preinvexity of order $m$ w.r.t. $\eta$ on geodesic invex sets, strongly $\eta$-invex functions of order $m$ and strongly invariant $\eta$-monotonicity of order $m$ on RM have been introduced. The definitions presented in this paper are supported by non-trivial examples.   Several interesting properties have also been discussed. An interesting application to $MOP$ for strongly $\eta$-invex functions of order $m$ has been presented and the characterization of a strict $\eta$-minimizers of order $m$ for $MOP$ and a solution to $VVLIP$ has been derived. Our results generalize the previously known results proven by different authors.

       \end{document}